\documentclass[leqno]{amsproc}
\usepackage{amsmath,amsthm,amssymb,amsfonts}
\usepackage{graphicx}
\usepackage{color}
\usepackage[active]{srcltx}
\usepackage{multicol}
\usepackage{soul}

\newtheorem{theorem}{Theorem}[section]
\newtheorem{lemma}[theorem]{Lemma}

\newtheorem{proposition}[theorem]{Proposition}
\newtheorem{notation}[theorem]{Notation}

\newtheorem{corollary}[theorem]{Corollary}
\theoremstyle{definition}
\newtheorem{definition}[theorem]{Definition}
\newtheorem{remark}[theorem]{Remark}
\numberwithin{equation}{section}

\def\N{{\mathbb N}}

\def\R{{\mathbb R}}

\def\C{{\mathbb C}}
\def\co{\hbox{\rm co}}

\def\ext{\hbox{\rm Ext}}
\def\llll{{\longrightarrow}}
\def\sep{{ \ \ }}
\def\sem{{\ \ \ \ }}
\def\seg{{\ \ \ \ \ \ }}

\def\com#1{{``#1''}}
\title[The Bishop-Phelps-Bollob{\'a}s property] {A basis of $\R ^n$ with  good isometric properties \\ and some applications to denseness \\ of norm attaining operators}

\author[M.D. Acosta]{Mar\'{\i}a D. Acosta}
\address{Universidad de Granada, Facultad de Ciencias, Departamento de An\'{a}lisis Matem\'{a}tico, 18071 Granada, Spain}
\email{dacosta@ugr.es}
\author[J.L. D{\'a}vila]{Jos{\'e} L. D{\'a}vila}
\address{Departamento de Matem\'{a}tica,  Facultad de Ciencias, Universidad de Los Andes, M{\'e}rida, 5111 Venezuela}
\email{jolu6@correo.ugr.es, jldavila@ula.ve}

\thanks{The  first  author was  supported  by Junta de Andaluc\'{\i}a grant  FQM--185  and also by Spanish MINECO/FEDER grant  MTM2015-65020-P. The second named   author was  partially supported  by Junta de Andaluc\'{\i}a grant  FQM--185.}

\begin{document}
{\large

\begin{abstract}
	We characterize real Banach spaces $Y$ such that the pair $(\ell_\infty ^n, Y)$ has the Bishop-Phelps-Bollob\'as property for operators. To this purpose it is essential the use of an appropriate
	 basis of the domain space $\R^n$. As a consequence  of the mentioned  characterization, we provide examples of spaces $Y$ satisfying such property. For instance, finite-dimensional spaces, uniformly convex spaces,
	  uniform algebras and $L_1(\mu)$  ($\mu$ a positive measure) satisfy the previous property.
\end{abstract}

\maketitle
\baselineskip=.65cm

\section{Introduction}
   
 In 1961 Bishop and Phelps showed that for any  Banach space the subset of  norm attaining functionals is dense in the topological dual \cite{BP}. These authors  posed the problem of possible extensions of such result to operators. In 1963 the pioneer work of Lindenstrauss provided some results  for  the vector valued case \cite{Lin}. He showed that in general the set of norm attaining operators is not dense in the corresponding space of bounded and linear operators and also proved some positive results. The survey \cite{Ac1} contains some interesting results about the topic and the state of the art until 2006.
 
 Throughout this paper $X^*$ is the topological dual of a normed space $X$. By $B_{X}$ and $S_{X}$ we denote the closed unit ball and the unit sphere of $X$, respectively. The symbol $L(X,Y)$ denotes the space of linear and bounded operators between two normed spaces $X$ and $Y$, endowed with the usual operator norm.
 
This paper deals with vector valued versions of the following assertion proved by Bollob{\'a}s in 1970 \cite{Bol}.

{\it Bishop-Phelps-Bollob{\'a}s Theorem} (\cite[Corollary 2.4]{CKMMR}). Let $X$ be a Banach space and $0< \varepsilon < 1$. Given $x \in B_X$  and $x^* \in S_{X^*}$ with $\vert 1- x^* (x) \vert < \frac{\varepsilon ^2 }{2}$, there are elements $y \in S_X$ and $y^* \in S_{X^*}$  such that $y^* (y)=1$, $\Vert y-x \Vert < \varepsilon$ and $\Vert y^* - x^* \Vert < \varepsilon $.
\vskip3mm

Let us mention that the previous  result has been applied to obtain properties for numerical ranges  of operators (see for instance \cite[\S 17]{BoDu}). Bishop-Phelps-Bollob{\'a}s theorem  and vector valued versions of Bishop-Phelps theorem  motivated the study of extensions of Bollob\'as  result 
for operators. In 2008 Acosta, Aron, Garc{\'i}a and Maestre introduced the following notion.
 
 \begin{definition} \label{def-BPBp}
	(\cite[Definition 1.1]{AAGM1}).  Let $X$  and $Y$ be both either real or complex Banach spaces. The pair $(X,Y )$ is said to have the Bishop-Phelps-Bollob{\'a}s property for operators (BPBp) if for every $  0 < \varepsilon  < 1 $  there exists $ 0< \eta (\varepsilon) < \varepsilon $ with the following property:
	\newline
	If $T\in S_{L(X,Y)}$ and $x_0 \in S_X$ satisfy $ \Vert T (x_0) \Vert > 1 - \eta (\varepsilon)$, then 
	there exist  $S \in S_{L(X,Y )}$ and $u_0 \in S_X$ satisfying the following conditions
	$$
	\Vert S (u_0) \Vert =1, \sem \Vert u_0- x_0 \Vert < \varepsilon \seg \text{and} 
	\sem \Vert S-T \Vert < \varepsilon.
	$$
\end{definition}

The same authors provided the first results on the topic showing that  the pair $(X,Y)$ has the BPBp when $X$ and $Y$ are finite-dimensional spaces  \cite[Proposition 2.4]{AAGM1}. 
%
The same result also holds  for any Banach space $X$ in case that    the Banach space $Y$ has a certain isometric property (called property $\beta$ of Lindenstrauss) \cite[Theorem 2.2]{AAGM1}. For instance, polyhedral finite-dimensional spaces, $c_{0}$ and $\ell_{\infty}$ have such property. There is also  a characterization of the spaces $Y$  such that the pair $(\ell_{1},Y)$ has the BPBp \cite[Theorem 4.1]{AAGM1}.

In case that $Y$ has the Radon-Nikod{\'y}m property, Choi and Kim proved that the pair $(L_{1}(\mu),Y)$ has the BPBp for any  $\sigma$-finite measure $\mu$   \cite[Theorem 2.2]{CK2}.  The result stating that  for any positive measures $\mu $ and  $\nu$ the pair $(L_1(\mu), L_1 (\nu))$  has the Bishop-Phelps-Bollob\'as property for operators is due to Choi, Kim, Lee and Mart{\'i}n   \cite[Theorem 3.1]{CKLM}.  The same authors also proved that  the pair  $(L_1(\mu), L_\infty (\nu))$ has the BPBp  for
	any  $\mu$ whenever   $\nu$ is a
	 localizable measure   \cite[Theorem 4.1]{CKLM}. This  result  was shown  before for the pair $(L_1(\mu), L_\infty [0,1])$  by Aron, Choi, Garc{\'i}a  and Maestre \cite{ACGM}. 
In case that $X$ is uniformly convex, then the pair $(X,Y)$ has the Bishop-Phelps-Bollob{\'a}s property for operators  for any Banach space $Y$, a result proved independently in  \cite[Theorem 3.1]{KLc} and \cite[Theorem 2.2]{ABGMt}.

Despite the fact that many authors proved  interesting results about this topic in the last years, it is not known whether or not the pair $(c_{0},\ell_{1})$ has the BPBp in the real case. In case that the domain is $c_0$, Kim proved that the pair  $(c_{0},Y)$ has the BPBp whenever $Y$ is a uniformly convex space \cite[Corollary 2.6]{KiI}. As a consequence of results due to Aron, Cascales and Kozhushkina, the pair $(X,C(K))$ has the BPBp if $X$ is an Asplund space  \cite[Corollary 2.6]{ACK}. Indeed this result was 
 extended to uniform algebras 
 by Cascales, Guirao and Kadets  \cite[Theorem 3.6]{CGK}.  As a consequence, the pair $(c_0,Y)$ has the Bishop-Phelps-Bollob\'as for operators for any uniform algebra $Y$. It is also known
   that in the real case the pair $(C(K),C(S))$ has the BPBp for any compact Hausdorff  spaces $K$ and $S$ \cite[Theorem 2.5]{ABCCetc}.  
 Kim and Lee showed  that the pair $(C(K),Y)$ has the BPBp whenever $Y$ is a uniformly convex space,   for any compact Hausdorff space $K$ 
 \cite[Theorem 2.2]{KL}, a result previously obtained in case that the domain is $L_{\infty}(\mu)$     for any positive measure $\mu$ \cite[Theorem 5]{KLL}.
  In
    the complex case  Acosta proved that  
  the pair $(C_{0}(L),Y)$ has the BPBp whenever $Y$ is $\C$-uniformly convex,   for any locally compact Hausdorff  space  $L$ \cite[Theorem 2.4]{AcB}. As a consequence,  in the complex case $(c_0,L_{1}(\mu))$ has the BPBp for any positive measure $\mu$.
However, as we already mentioned above, in the real case  it is an open problem whether or not the parallel result   holds true  even for the pair $(c_{0},\ell_{1})$. Let us point out that the set of  norm attaining operators from $c_0$ to $\ell_1$ is dense in $L(c_{0},\ell_{1})$  since  every operator from $c_0 $ to $\ell_1$ is compact and the usual basis of $c_0$ is monotone and shrinking.

As a consequence of \cite[Theorem 2.1]{ACKLM}, if the pair $(c_{0},\ell_{1})$ has the BPBp then
the pairs $(\ell_\infty ^n, \ell_1)$ satisfy the BPBp \com{uniformly} for every $n$.
It is not difficult to check that the  converse  also holds.  Those facts  motivated  our \com{finite-dimensional} approach.  So we began to study for which   natural numbers $n$  the pair  $(\ell_\infty ^n, \ell_1)$ has the BPBp. Indeed we posed the  more general question whether or not, given a space $Y$, the pair $(\ell_\infty ^n,Y)$  has the BPBp.
For $n=1$  this condition is trivially satisfied  since the pair $(\R,\R)$ has the BPBp. For $n=2$, there exists a characterization of the  Banach spaces $Y$  such that $(\ell_\infty ^2, Y)$ has the BPBp. In case $n=3$ and $n=4$ there are two papers containing  characterizations of the spaces $Y$ such that the pair $(\ell_\infty ^n, Y)$ has the BPBp (see \cite[Theorem 2.9]{ABGKM2}  and \cite[Theorem 3.3]{ADS}).   The space $Y=\ell_1$ satisfies that assertion
\cite[Theorem 3.3 and Corollary 4.8]{ADS}.
 
The goal of this paper is to  characterize  the  Banach spaces $Y$ such that the pair $(\ell_\infty ^n, Y)$ has the BPBp for  some  fixed positive integer  $n$ (Theorem \ref{teo-char}). For this purpose we introduce  in Section 2 an isometric property, the so-called approximate  hyperplane  sum property for $\ell_\infty^n$ (AHSp-$\ell_\infty^n$ for short, see Definition \ref{def-AHSPLn}). We also provide several conditions that are equivalent to the  AHSp-$\ell_\infty^n$ (Proposition \ref{pro-char}).   The main result of Section 3 states that  the pair $(\ell_\infty ^n,Y) $ has the Bishop-Phelps-Bollob\'{a}s property for operators if and only if $Y$ has the AHSp-$\ell_\infty^n$.    Section 4 contains 
examples of spaces  having the AHSp-$\ell_\infty^n$. As a consequence of previous results and the characterization provided in Section 3,  finite-dimensional spaces, uniformly convex spaces and   uniform algebras enjoy  AHSp-$\ell_\infty^n$, for any natural number $n$.
 In the case of $\ell_1$ we prove directly that this space also enjoys such isometric property  (see Theorem \ref{l1-AHSPLn}).
 From this we easily deduce that the space $L_1(\mu)$ also satisfies the 
AHSp-$\ell_\infty^n$,  for any natural number $n$ and any positive measure $\mu $ (Corollary  \ref{cor-L1-AHSP-lin}). As a consequence,  each space  $Y$  belonging to one of the classes mentioned  above satisfies that the pair $(\ell_\infty ^n , Y ) $ has the BPBp, for any natural number $n$.

Now we will explain some of the key ideas used throughout this paper.
One of the  first basic ideas is to identify the Banach spaces $L(\ell_\infty ^n,Y)$ and $Y^n$,  endowed with some appropriate norm (Proposition \ref{prop-ident}). Under this identification the   closed unit ball of  $L(\ell_\infty ^n,Y)$ is associated to the set denoted  by $M_{Y}^n$ (Notation \ref{not-MnY}).  For  this identification 
we use a  \com{simple} basis (denoted by $\mathcal{B}_n$)
of the domain space $\R^n$
 whose  elements are extreme points of  $ B_{\ell_\infty ^n}$. This basis also satisfies, for instance,  that  any extreme point of $ B_{\ell_\infty ^n}$  may be expressed easily in terms of  $\mathcal{B}_n$. 
  In this way we reformulate  the BPBp
for the pair $(\ell_\infty ^n,Y)$ in terms of an  intrinsec condition on $Y$
involving the set $M_{Y}^n$ (Theorem \ref{teo-char}). That is, the space $Y$ has the   approximate  hyperplane  sum property for $\ell_\infty^n$.   In order to prove Theorem \ref{teo-char} it is essential  to characterize  the isometric property AHSp-$\ell_\infty ^n$ 
in differente ways (Proposition \ref{pro-char}). Let us point out that the proof of Theorem \ref{teo-char} is simpler   than the previous characterization for the particular case of $n=4$  provided in \cite{ADS}.

In order to obtain applications of the main result,  we used Theorem \ref{teo-char}  and the known classes of  spaces $Y$ such that $(\ell_\infty ^n, Y)$ has the BPBp. Also   condition 3) in Proposition \ref{pro-char}
is quite useful  and indeed this is the key idea used to prove that $\ell_1$ satisfies the AHSp-$\ell_\infty ^n$. The proof of this fact is technical and long, and uses induction.  In any case that proof  for the general case  is much easier that the proof of the particular case $n=4$
given in \cite{ADS}.  From the result for $\ell_1$  we easily deduce  that $L_1 (\mu)$ also satisfies the AHSP-$\ell_\infty ^n$, for any positive measure $\mu$ and any positive integer $n$. Hence  the pair  $(\ell_\infty ^n, L_1(\mu))$ satisfies the Bishop-Phelps-Bollob\'as property for operators.

Throughout this paper   we  consider only \textit{real} normed spaces.  The symbol $Y$ denotes  a real normed space and $n$ is a  fixed nonnegative integer. 
We denote by $\ell _ \infty ^n$ the space $\R^n,$ endowed with the norm given by $\Vert x \Vert = \max \{ \vert x_i \vert : i \le n \}.$



\section{The approximate hyperplane sum property for $\ell_\infty ^n$}

The goal of 
this section is to  introduce and characterize  the  approximate hyperplane sum property for $\ell_\infty ^n$.    For such purpose  we begin with some  notation and    simple technical results.  

Firstly  we introduce sets that   will play an essential role in the characterization of the Bishop-Phelps-Bollob\'{a}s property for operators for   a pair $(\ell_\infty ^n, Y)$.
We define the sets $\mathcal{I}_n$ and $\mathcal{P}_n$ by
$$ 
\mathcal{I}_n:=\{ (i_1, \ldots ,i_k)\in \{1,\ldots, n\}^k : k \sep \text{odd}, \sep k\le n \sep \text{and} \sep i_j < i_{j+1},  \, \forall   j<k \}
$$
and  
$$
\mathcal{P}_n:=\{ (i_1, \ldots ,i_k)\in \{1,\ldots, n\}^k : k \sep \text{even}, \sep k\le n \sep \text{and} \sep i_j < i_{j+1},  \,  \forall   j<k \}.
$$

\begin{notation}
		\label{not-MnY}
	If $Y$ is a Banach space and $n$ is a positive integer, we define
	$$
	M^n_Y:= \bigg\{(y_i)_{i \le n}  \in Y^n:  \sum_{j=1}^{k}(-1)^{j+1}y_{i_j}  \in B_Y, \sep \forall   (i_1, \ldots ,i_k) \in \mathcal{I}_n  \bigg\}.
	$$
It is clear that $M^n_Y$ is a subset of $(B_Y)^n$. For each  $1\le i \le n$ we define the vector $v_i^{n}\in \ell_\infty^n $ as follows
  $$
    v_i^{n}(j) :=
    \begin{cases}
    -1           & \mbox{if} \sep 2\le j \le i  \\
    1  & \mbox{if} \sep  j=1 \sep \text{or} \sep i+1 \le j \le n.
    \end{cases}
    $$
    We also denote  by $\mathcal{B}_n, E_{1}^n$ and  $O_n$ the sets given by
$$
\mathcal{B}_n:=\{ v_i^{n}: 1 \le i \le n\},
$$
$$
E_1^{n}:= \{v\in \ell_\infty^n: v(1)=\| v \|=1 \}
$$
and 
$$
O_n:=\{v \in E_1^{n}: v(i)\le v(i+1) \sep \forall 2\le i < n \}.
$$
\end{notation}

The set  $\mathcal{B}_n$ plays an essential role in this paper.  Besides  it gives a nontrivial example of an element of $M^n_{\ell_\infty^n}$  with  interesting properties such as those that  we   prove next.

\begin{lemma}
	\label{B-base}
	If  $ n \ge 2$ and  $v\in \R^n$ then $v$ can be expressed as follows   
	$$v= \frac{v(1)+v(2)}{2} v_1^{n} + \sum _{i=2}^{n-1} \frac{v(i+1)- v(i)}{2} v_i^{n} +  \frac{v(1)-v(n)}{2} v_n^{n}.$$
	In particular, then set  $\mathcal{B}_n$ is a basis for $\R^n$.
\end{lemma}

\begin{proof}
Let $v\in \R^n$ and $1\le j \le n$. Firstly suppose that  $j=1$; in this case we have 
\begin{eqnarray*}
v(1)&=&\frac{v(1)+v(2)}{2} + \frac{v(n)-v(2)}{2}+ \frac{v(1)-v(n)}{2}\\
&=&\frac{v(1)+v(2)}{2} + \sum _{i=2}^{n-1} \frac{v(i+1)- v(i)}{2} +  \frac{v(1)-v(n)}{2}\\
&=&\frac{v(1)+v(2)}{2} v_1^{n}(1) + \sum _{i=2}^{n-1} \frac{v(i+1)- v(i)}{2} v_i^{n}(1) +  \frac{v(1)-v(n)}{2} v_n^{n}(1).\\
\end{eqnarray*}

In the case that $j>1$ we have
\begin{eqnarray*}
v(j)&=&\frac{v(1)+v(2)}{2} + \frac{-v(2)+v(j)}{2}+ \frac{-v(j)+v(n)}{2}(-1)+ \frac{-v(1)+v(n)}{2}\\
&=&\frac{v(1)+v(2)}{2} + \sum _{i=2}^{j-1} \frac{v(i+1)- v(i)}{2}- \sum _{i=j}^{n-1} \frac{v(i+1)- v(i)}{2} -  \frac{v(1)-v(n)}{2}\\
&=&\frac{v(1)+v(2)}{2} v_1^{n}(j) + \sum _{i=2}^{n-1} \frac{v(i+1)- v(i)}{2} v_i^{n}(j) +  \frac{v(1)-v(n)}{2} v_n^{n}(j).
\end{eqnarray*}
\end{proof}

It is clear that $\mathcal{B}_n\subset O_n$. Since $O_n$ is convex we get the following inclusion
$$
\co(\mathcal{B}_n)\subset O_n,
$$
where   $ \co (A)$ denotes the convex hull of a set $A \subset \R ^n$. 
From  Lemma  \ref{B-base} we  obtain easily that every element of  $O_n$ is a convex combination of elements in  $\mathcal{B}_n$. As a consequence, we have the following result.

\begin{lemma}
	\label{coB=O}
	For each $n\in\N$ we have that
	$$\co(\mathcal{B}_n)=O_n.$$
\end{lemma}

By the previous equality  and the definition of $O_n$ the convex hull of  $\mathcal{B}_n$  has a very nice description.  There is  another reason that makes the  set $\mathcal{B}_n$ special.
 These base also have the property that the image under the  projection onto the first $n$ coordinates of $\mathcal{B}_{n+1}$ coincides with $ \mathcal{B}_{n}$. This fact  will be used to prove by induction  that the elements  in $\ext(E_1^{n})=\{ v \in E_1^{n}: |v(i)|=1  \sep \forall 1<i\le n\} $ can be expressed easily in terms of the basis $\mathcal{B}_n$.

\begin{lemma}
	\label{ExE1}
	For each natural number $n$  we have that
	$$
	\ext(E_1^{n})=\bigg\{ \sum_{j=1}^{k}(-1)^{j+1}v_{i_j}^{n}:   (i_1, \ldots ,i_k) \in \mathcal{I}_n\bigg \}.
	$$
\end{lemma}

\begin{proof}
We will begin by showing the inclusion
$$
\bigg\{ \sum_{j=1}^{k}(-1)^{j+1}v_{i_j}^{n}:   (i_1, \ldots ,i_k) \in \mathcal{I}_n\bigg \} \subset \ext(E_1^{n})
$$

Let be $(i_1, \ldots ,i_k)$ an element in $\mathcal{I}_n$. We  denote by $z=\sum_{j=1}^{k}(-1)^{j+1}v_{i_j}^{n}$ and for every $1\le j_0\le n$ consider the sets
given by
$$
A_{j_0}:=\{s\le k : v_{i_s}^n(j_0)=1\} \,\,\,  \text{and} \,\,\, B_{j_0}:=\{s\le k : v_{i_s}^n(j_0)=-1\}.
$$
Since 
$A_{j_0}$ and $B_{j_0}$ are disjoint subsets  whose union is the set $\{1,\ldots, k\}$,   we have that 
$$
z(j_0)= \sum_{j\in A_{j_0}}(-1)^{j+1}1 + \sum_{j\in B_{j_0}}(-1)^{j+1}(-1).
$$ 
By using also that $A_{j_0}$ and $B_{j_0}$ are  intervals of $\{s \in \N : s \le k\}$ and  $k$  is odd we obtain that  $z(j_0) \in \{1, -1\}$. Moreover, it is clear that  $z(1)=1$ since $v_i^n(1)=1 $ for  each $i \le n$  and $k$ is odd.  
Hence $z \in E_1 ^n  \cap   \ext( B _ {\ell _\infty ^n } ) $, so $z$ is an extreme point of $E_1^{n}$.

Our aim now is to prove the inclusion
\begin{equation}
\label{E1-E}
\ext(E_1^{n})\subset \bigg\{ \sum_{j=1}^{k}(-1)^{j+1}v_{i_j}^{n}:   (i_1, \ldots ,i_k) \in \mathcal{I}_n\bigg \}.
\end{equation}
We prove the previous assertion by  induction.  For $n=1$  condition \eqref{E1-E} is trivially satisfied.  Assume that \eqref{E1-E}  holds for a natural number $n$.  If  $v \in \ext(E_1^{n+1})$, we define  an element  $v' $ in $\R^n$ by
$$
v'(i)=v(i),  \sep \forall i\le n.
$$
Clearly $v' \in  \ext(E_1^{n})$. 
By assumption there exists $(i_1, \ldots ,i_k) \in \mathcal{I}_n$ such that 
$$
v'=\sum_{j=1}^{k}(-1)^{j+1}v_{i_j}^{n}.
$$
From this  we obtain 
$$
    v=
    \begin{cases}
    \sum_{j=1}^{k}(-1)^{j+1}v_{i_j}^{n+1}        & \mbox{if} \sep v(n+1)=1  \\
    \sum_{j=1}^{k-1}(-1)^{j+1}v_{i_j}^{n+1} + v_{n+1}^{n+1}  & \mbox{if} \sep  v(n+1)=-1 \sep \text{and} \sep i_k=n\\
    \sum_{j=1}^{k}(-1)^{j+1}v_{i_j}^{n+1} - v_n^{n+1} + v_{n+1}^{n+1} &\mbox{if} \sep  v(n+1)=-1 \sep \text{and} \sep i_k<n.
    \end{cases}
$$
The previous expression shows that  $v\in \biggl\{  \displaystyle{\sum_{p=1}^{m}}(-1)^{p+1}v_{j_p}^{n+1}:   (j_1, \ldots ,j_m) \in \mathcal{I}_{n+1}\biggr \}$.
\end{proof}

The following  result identifies the closed unit ball of the  space $L(\ell _\infty ^n, Y)$     with  $M_Y ^n$.
  It  extends \cite[Proposition 2.11]{ADS}.

\begin{proposition}
	\label{prop-ident}
	The mapping   $\Phi : L(\ell_\infty ^n, Y) \llll Y^n $ given by $\Phi (T)= ( T(v^n_i))_{i \le n} $ is a linear bijection. It is also satisfied that 
	$$
	\Vert T \Vert = \max \bigg \{ \biggl\Vert T \bigg(\sum_{j=1}^{k}(-1)^{j+1}v_{i_j}^n\bigg)\biggr\Vert: (i_1, \ldots ,i_k) \in \mathcal{I}_n\bigg\}.
	$$
In particular, $\Phi(B_{ L(\ell_\infty ^n, Y)})=M_Y ^n.$ 
\end{proposition}
\begin{proof}
	By Lemma \ref{B-base}  the set $\mathcal{B}_n=\{ v_i ^{n}: i \le n\}$ is a basis of $\R^n$, so  every operator $T\in L(\ell_\infty ^n, Y)$  is determined by the element $( T(v_i^{n}))_{ i \le n}\in Y^n$.  From this it is immediate that $\Phi$ is a linear bijection. 
	
It is clear that  
$$
\ext \bigl( B_{\ell_\infty ^n}\bigr)  = \{ v \in B_{\ell_\infty ^n}: |v(i)|=1,  \sep \forall 1 \le i\le n\}=\ext \bigl(E_1^{n} \bigr)  \cup \ext \bigl(-E_1^{n}\bigr).
$$  
Then by Lemma \ref{ExE1} we have the following
\begin{eqnarray*}
\Vert T \Vert&=& \max \{ \Vert T(e) \Vert : e \in \ext \bigl( B_{\ell_\infty ^n }\bigr) \} \\
&=& 	\max \{ \Vert T(e) \Vert : e \in \ext (E_1^{n}) \} \\
&=& \max \bigg \{ \biggl\Vert T \bigg(\sum_{j=1}^{k}(-1)^{j+1}v_{i_j}^n\bigg)\biggr\Vert: (i_1, \ldots ,i_k) \in \mathcal{I}_n\bigg\}.
\end{eqnarray*}
From the previous expression it follows that $T \in B_{L(\ell_\infty ^n, Y)}$ if and only if $\Phi(T) $ is an element in $M_Y ^n$.
\end{proof}

\begin{definition}
	\label{def-AHSPLn}
	Let $n$  be a positive integer. A Banach space $Y$  has the  {\it  approximate  hyperplane  sum property for $\ell_\infty^n$} (AHSp-$\ell_\infty^n$)  if for every $0 < \varepsilon < 1$ there is $ 0< \gamma_n (\varepsilon) < \varepsilon $ satisfying the following condition
	
	For every $ (y_i)_{i \le n}  \in M_Y^n$, if there exist a nonempty subset  $A$ of $\{1,\ldots, n\}$ and $y^* \in
	S_{Y^*}  $ such that $y^* (y_i) > 1 - \gamma_n (\varepsilon) $ for each $i \in A$, then there exists an element $ (z_i)_{i \le n}  \in M_Y^n$  satisfying $\Vert z_i - y_i \Vert < \varepsilon $ for every $i \le n$ and $\Vert \sum_{i \in A } z_i \Vert = \vert A\vert $.
\end{definition}

\begin{notation}
		\label{not-tau}
	By $\tau_n$ we  denote the  mapping on $Y^n$ given by
$$
	    \tau_n (y_1,\ldots, y_n):=(y_2, \ldots,y_n, -y_1).
$$
\end{notation}

It is clear that $\tau_n$  is a bijective  mapping from $Y^n$ onto itself. 
\noindent 
For  a fixed element  $(y_1, \ldots, y_n ) \in Y^n$ we denote  by  $(y'_1, \ldots, y'_n) = (y_2, \ldots,y_n, -y_1)$.
If $(i_1, \ldots ,i_k) \in \mathcal{I}_n$ we have the following   identities
$$\sum_{j=1}^{k}(-1)^{j+1}y'_{i_j}=\sum_{j=1}^{k}(-1)^{j+1}y_{i_j+1} \sep \text{if} \sep  i_k < n
$$
$$
\sum_{j=1}^{k}(-1)^{j+1}y'_{i_j}=-(y_1 + \sum_{j=1}^{k-1}(-1)^{j}y_{i_j+1})  \sep \text{if} \sep  i_k = n.
$$
From the previous equalities it follows that $\tau_n  (M_Y^n) \subset M_Y^n $. By  using also that   $\tau_n^{2n}$ is the  identity mapping  on $Y^n$, we obtain the following result.

\begin{lemma}
	\label{le-y-y1}
Let $(y_i) _{i \le n} \in Y^n$ and $m\in \N$. The following assertions are equivalent
	\begin{enumerate}
		\item[1)]  $(y_i)_{i \le n}  \in M_Y^n .$
		\item[2)]  $\tau_{n}^m((y_i)_{i \le n})  \in M_Y^n.$
	\end{enumerate}
\end{lemma}

In order to provide   a characterization of the AHSp-$\ell_\infty^n$  that will be essential  in the rest of the paper, we recall the following notion. A subset $B\subset B_{Y^{*}}$ is $1$-\textit{norming} if
$$
\|y\|=\sup \{\vert y^{*}(y) \vert : y^{*}\in B\},  \sep  \forall   y\in Y.
$$

\begin{proposition}
    \label{pro-char}
    Let $Y$ be a Banach space, $n$ a positive integer and $B \subset S_{Y^*}$  a  $1$-norming subset. The following assertions  are equivalent.
    \begin{enumerate}
       
       \item[1)]  $Y$  has the AHSp-$\ell_\infty^n$.
       
       \item[2)] The condition stated in Definition \ref{def-AHSPLn} is satisfied for each $y^*\in B$.      
        \item[3)]  For every $0 < \varepsilon < 1$ there exists $ 0< \rho_n (\varepsilon) < \varepsilon $ such that for every element $ (y_i)_{i \le n}  \in M_Y^n$,  if there exist $n_0\le n$ and $y^* \in B$ such that $y^* (y_i) > 1 - \rho_n (\varepsilon) $ for each $i \le n_0$, then there exists an element $ (z_i)_{i \le n}  \in M_Y^n$  satisfying 
        $$\Vert z_i - y_i \Vert < \varepsilon \sep \text{for each}  \sep i \le n \sep \text{and} \sep \Big\Vert \sum_{i=1}^{n_0} z_i \Big\Vert = n_0.$$
        
        \item[4)]
        For every $0< \varepsilon <1$ there exists  $0 < \nu_n( \varepsilon)<  \varepsilon $ such that for each element  $(y_i)_{i \le n} \in M_Y^n$  and each convex combination  $\sum_{i=1}^n \alpha_i y_i $ satisfying
        $$
        \Bigl \Vert \sum_{i=1}^n \alpha_i y_i\Bigr \Vert  > 1-\nu_n(\varepsilon ),
        $$
        there exist a set $C\subset \{1,\ldots, n\}$ and an element $(z_i)_{i \le n} \in M_Y ^n$  such that
        \begin{enumerate}
            \item[i)] $\sum_{i\in C}\alpha_i>1-\varepsilon,$
            \item[ii)]  $\|z_i-y_i\|<\varepsilon  \sep \text{for each} \sep i \le n$ \  and   
            \item [iii)] $\Vert \sum _{i \in C } z_i \Vert = \vert C \vert$.
        \end{enumerate}
    \end{enumerate}
    
Moreover, if $\gamma_n $ is  a function satisfying 1) (see Definition \ref{def-AHSPLn})  then condition 4) holds  for $\nu_n = \gamma_n ^2$. Condition 4) for  a function $\nu_n $  implies  1)  for  the function $\gamma_n$ given by    $\gamma_n(\varepsilon)=\nu_n(\frac{\varepsilon}{n}) $. 
In case that 3) holds for  $\rho_n$ then 1) is satisfied for  the function given by  $\gamma_n(\varepsilon)= \frac{1}{4}\rho^2_n \bigl( \frac{\varepsilon}{n} \bigr)$.                       
\end{proposition}
\begin{proof}
    Clearly 1) implies 2) and 2) implies 3).\\
\noindent
3) $\Rightarrow$ 2)\\
Let $0<\varepsilon<1$ and $\rho_n(\varepsilon)$  be the positive real number satisfying  condition 3). We put $\gamma ^\prime _n(\varepsilon)=\frac{\rho_n(\varepsilon)}{2}$. Let $ (y_i)_{i \le n}  \in M_Y^n$ and assume that there exists a nonempty subset $A$ of $\{1,\ldots, n\}$ and $y^* \in B $ such that 
$$
y^* (y_i) > 1 - \gamma ^\prime _n (\varepsilon),  \sep \forall i \in A.
$$
By Lemma \ref{le-y-y1}  we  may   assume without loss of generality that $1\in A$. On the other hand, if $i,k \in A$ and $j$ is an integer such that $i<j<k$, then
$$
2-2\gamma ^\prime _n(\varepsilon)-y^{*}(y_j) < y^{*}(y_i-y_j+y_k)\le \|y_i-y_j+y_k\| \le 1.
$$
As a consequence,  $y^{*}(y_j)> 1- \rho_n(\varepsilon)$. Hence we obtain that
$$
y^* (y_i) > 1 - \rho_n (\varepsilon) , \sep \forall i \le \max A.
$$
By assumption there exists $ (z_i)_{i \le n}  \in M_Y^n$  satisfying  
\begin{enumerate}
        \item[i)] $\Vert z_i - y_i \Vert < \varepsilon$ for each $i \le n$ and 
        \item[ii)]  $\Vert \sum_{i=1}^{\max A} z_i \Vert = \max A.$
          \end{enumerate}
          Lastly, by Hahn-Banach Theorem  and the fact that $\{ z_i : i \le n \} \subset B_Y$,    condition ii) implies  that $\Vert \sum_{i\in A} z_i \Vert = \vert A\vert.$      

\noindent
2) $\Rightarrow$ 4)
\newline
    Assume that $Y$ satisfies condition 2).  For each $0<\varepsilon < 1$  let $ \gamma ^\prime _n(\varepsilon)<\varepsilon$  be the positive real number satisfying Definition \ref{def-AHSPLn} for every element $y^{*}\in B$. We take $\nu_n(\varepsilon)=\gamma ^\prime _n(\varepsilon)^2$.

    Let   $ (y_i)_{i \le n} \in M_Y^n$ and assume that the convex combination $\sum_{i=1}^{n} \alpha_i y_i$   satisfies  
     $\bigl\Vert  \sum_{i=1}^{n} \alpha_iy_i \bigr \Vert  > 1 - \nu_n(\varepsilon) $. Since  $B$ is a $1$-norming set and  $(-y_i)_{i \le n} \in M_Y^n$, by using $(-y_i)_{i \le n}$ instead of $(y_i) _{i \le n}$, if needed,     there is  $y^{*}\in B$ such that
    $$
   y^{*}\biggl( \sum_{i=1}^{n} \alpha_i y_i \biggr) =   \Bigl \Vert   \sum_{i=1}^{n} \alpha_iy^{*}(y_i)  \Bigr \Vert >  1 - \nu_n(\varepsilon) = 1 - \gamma ^\prime _n(\varepsilon) ^2 .
    $$
    By \cite[Lema 3.3]{AAGM1} the set $C:=\{ i\le n : y^*(y_i)> 1-\gamma ^\prime _n(\varepsilon) \}$
    satisfies
    $$
    \sum_{i\in C}\alpha_i \geq 1 - \frac{\nu_n(\varepsilon)}{\gamma ^\prime _n(\varepsilon)}> 1- \varepsilon.
    $$
    By assumption there is an element $(z_i )_{i \le n}  \in M_Y^n$  such that  $\| z_i - y_i\| < \varepsilon$ for each  $i \leq n$ and 
     \linebreak[4]
    $\|\sum_{i \in C} z_i \| = |C|$.

    \noindent
    4) $\Rightarrow$ 1)
    \newline
Now we assume that $Y$ satisfies condition  4).
     Given  $0<\varepsilon < 1$, let  $\nu_n(\varepsilon)$ be the positive real number satisfying the assumption. We will show that  $\gamma_n(\varepsilon)=\nu_n(\frac{\varepsilon}{n})$ satisfies Definition \ref{def-AHSPLn}.

    Let $(y_i)_{i \le n} \in M_Y^n$  and assume that for some nonempty set  $A \subset \{1,\ldots, n\}$ and $y^{*} \in S_{Y^{*}}$ it is satisfied that $y^{*}(y_i)>1-\gamma_n(\varepsilon)$ for each  $i\in A$. We define the following nonnegative real numbers
    $$
    \alpha_i =
    \begin{cases}
    \frac{1}{|A|}            & \mbox{if} \sep i  \in A   \\
    0  & \mbox{if} \sep  i \in \{1,\ldots, n\} \backslash  A.
    \end{cases}
    $$
    Clearly $\sum _{i=1}^ n \alpha _i =1$ and we also have that
    $$
    \Big\| \sum_{i =1}^{n}\alpha_iy_i \Big\| = \frac{\Big\| \sum_{i' \in A}y_i \Big\| }{|A|}\geq \frac{y^{*}\Big( \sum_{i' \in A}y_i \Big )}{|A|}> 1-\nu_n\Big(\frac{\varepsilon}{n}\Big).
    $$

    By assumption there is a set $C\subset \{ 1,\ldots,n\}$ and $(z_i)_{i\le n} \in  M_Y^n$  such that
    \begin{enumerate}
        \item[i)] $\sum_{i\in C}\alpha_i>1-\frac{\varepsilon}{n},$
        \item[ii)]  $\|z_i-y_i\|<\frac{\varepsilon}{n}  \sep \text{for each}\sep  i \le n$ and 
        \item [iii)]  $ \bigl \Vert \sum _{ i \in C} z_i \bigr\Vert  = \vert C \vert $.
    \end{enumerate}
In case that $A \subset C$, condition iii) and Hahn-Banach Theorem implies that 
$ \bigl \Vert \sum _{ i \in A} z_i \bigr\Vert  = \vert A \vert $. So it suffices to prove $A \subset C$.
If  there  were  some $i_0 \in A\setminus C$,  we put  $B= \{i \in \{1,\ldots, n\}: i \ne i_0\}$  and so by using i) we have that 
    $$
    1- \frac{1}{\vert A \vert} =\sum_{ i \in B} \alpha_i\geq \sum_{i\in C} \alpha_i >1-\frac{\varepsilon}{n} >1- \frac{1}{n}.
    $$
    Then  $\vert A \vert >  n$, which is a contradiction. Hence, $A\subset C$ and we proved that $Y$ has the AHSp-$\ell_\infty^n$.
    
    As a consequence of the proofs of   3) $\Rightarrow$ 2),   2) $\Rightarrow$ 4) and  4) $\Rightarrow$ 1),  we deduce that a space $Y$ satisfying 3) for the function   $\rho_n$ also has the AHSp-$\ell_\infty ^n$ for the function  given by  $\gamma_n(\varepsilon)= \frac{1}{4}\rho^2_n \bigl( \frac{\varepsilon}{n} \bigr)$. 	
\end{proof}

Notice that Proposition  \ref{pro-char} makes  easier to show that a space  has the AHSp-$\ell_\infty^n$. To this purpose condition 3) is useful since it suffices to check Definition  \ref{def-AHSPLn} only for functionals in a $1$-norming set and  for simpler sets $A$.

\begin{remark}
	\label{re-n+1-n}
 AHSp-$\ell_\infty^{n+1}$ implies   AHSp-$\ell_\infty^{n}$ for each positive integer $n$. 
\end{remark}
This assertion can be easily checked  by using Definition  \ref{def-AHSPLn}.  It  follows from the following fact
$$
(y_i) _{ i \le n }  \in M_{ Y}^n \sep \Rightarrow \sep  (z_i) _{ i \le n+1 }  \in M_{ Y}^{n+1},
$$ 
where $z_i= y_i$ for $i \le n$ and $z_{n+1} = y_n$.

Notice also that Definition \ref{def-AHSPLn} is trivially satisfied for $n=1$. For each  $n \ge 2$ this is not the case. Indeed since $\ell_\infty ^2 $ and $\ell_1 ^2$ are isometric, in case that $Y$ is strictly convex and $n \ge 2$, by using  Remark \ref{re-n+1-n}, the proof of   \cite[Theorem 4.1]{AAGM1} and \cite[Lemma 3.2]{ACKLM}, if  $Y$ has the   AHSp-$\ell_\infty^{n}$   then $Y$ is uniformly convex.    By  the characterization that we will prove later  (Theorem \ref{teo-char})  and  \cite[Theorem 2.5]{KiI} the  converse result also holds since uniformly convex spaces $Y$ satisfy that the pair $ (\ell_\infty ^n , Y )$  has the Bishop-Phelps-Bollob\'{a}s property for operators for each $n\in \N$.

\section{ A characterization of the spaces $Y$  such that the pair $( \ell_\infty^n, Y)$ has the
Bishop-Phelps-Bollob{\'a}s property for operators}

The  main result of this section states that  Banach spaces $Y$ satisfying that the pair $(\ell _\infty ^n, Y)$ has the BPBp for operators  are  those having  the AHSp-$\ell_ \infty ^n$.  In order to prove such characterization it is useful to isolate  two techical results.
 Roughly speaking, the first  lemma states that an operator in $L(\ell_\infty ^n, Y)$ attaining its norm at a point $y$  of $S_{\ell_\infty ^n}$  that is close to an element, say $x$, belonging to one of the maximal faces of $B_{\ell_\infty ^n}$,  in fact also attains its norm at another  element of the same maximal face which is also  close to $x$.    In fact next assertion is more precise.

\begin{lemma}
\label{le-close-face}
Assume that $(\varepsilon, n)\in ]0,1[\times \N$, $S\in S_{ L(\ell_\infty ^n, Y)}$ and $(x,y) \in \co (\mathcal{B}_n) \times S_{\ell_\infty ^n}$ satisfies that 
$$
\Vert S(y)\Vert=1 \sep \text{and} \sep \Vert y-x\Vert < \varepsilon.
$$
Then there exists $z\in  \co (\mathcal{B}_n)$ such that
$$
\Vert S(z)\Vert=1 \sep \text{and} \sep \Vert z-x\Vert < \varepsilon.
$$
\end{lemma}
\begin{proof}
We begin by checking 
the following claim.

 \underline{Claim}: If $i_0 \le n$, $|y(i_0)| < 1$ and $t\in [-1,1]$ then the element   $y'\in S_{\ell_\infty ^n}$ defined by
  $$
    y'(i) =
    \begin{cases}
    y(i)            & \mbox{if} \sep i \neq i_0   \\
    t  & \mbox{if} \sep  i=i_0
    \end{cases}
  $$
  satisfies that $\Vert S(y')\Vert=1$.


It can be assumed  without loss of generality that $y(i_0)\in [t,1[$ (in  case that  $y(i_0)\in ]-1,t]$ we proceed analogously). 

 Let  define   the element   $ y''$ in $ \ell_\infty ^n  $
 	by
  $$
    y''(i) =
    \begin{cases}
    y(i)            & \mbox{if} \sep i \neq i_0   \\
    1  & \mbox{if} \sep  i=i_0,
    \end{cases}
  $$
  that clearly belongs to $S_{\ell_\infty ^n}$. 
There is $\alpha \in]0,1]$ such that
$
y(i_0)=\alpha t+(1-\alpha )1$ and so we have that $y=\alpha y'+(1-\alpha) y''$. Since $S$ attains its norm at $y$ we have that $\Vert S(y')\Vert=1$. So we proved the claim.

Now we consider the following sets
$$
A=\{i\in  \{2,\ldots, n\}: y(i)=-1\} \sem \text{and} \sem B=\{i\in  \{2,\ldots, n\}: y(i)=1\}.
$$
Note that if $i\in A$ and $j\in B$ the assumption $\Vert y-x \Vert < \varepsilon < 1$ implies that  $x(i)<0<x(j)$. Then by Lemma \ref{coB=O} we have that $i<j$  since  $x\in O_n$. Let define the following numbers 
$$ M_{-1}= \max (A \cup \{1\}) \sep \text{and} \sep m_{1}= \min (B \cup \{n+1\}).
$$

We know that  $M_{-1}<m_1$. Finally  the element  $z\in S_{\ell_\infty ^n}$ given  by

  $$
    z(i) =
    \begin{cases}
    1		& \mbox{if} \sep i=1   \\
    -1		& \mbox{if} \sep 2\le i \le M_{-1}   \\
    x(i)	& \mbox{if} \sep M_{-1}<i<m_{1}   \\
    1 		& \mbox{if} \sep m_{1}\le i \le n.
    \end{cases}
  $$
It is clear that  $z\in O_n$. Also,
$$
 \vert z(i) - y(i) \vert = |1-y(1)|=|x(1)-y(1)|<\varepsilon < 1,
$$
which implies that $y(1)>-1$.  By using the claim we get that $\Vert S(z) \Vert=1$.\\
On the other hand,
\begin{itemize}
\item If $2\le i \le M_{-1}$ then
\begin{eqnarray*}
|z(i)-x(i)|&=& |-1-x(i)|=1+x(i)\\
&\le& 1 + x(M_{-1}) \\
&=& |-1 - x(M_{-1})|=|y(M_{-1}) - x(M_{-1})| \\
&<& \varepsilon
\end{eqnarray*}
\item If $m_1\le i \le n$ then
\begin{eqnarray*}
|z(i)-x(i)|&=& |1-x(i)|=1-x(i)\\
&\le& 1 - x(m_{1}) \\
&=& |1 - x(m_{1})|=|y(m_{1}) - x(m_{1})| \\
&<& \varepsilon.
\end{eqnarray*}
\end{itemize}

Therefore $\| z- x\|< \varepsilon$.
\end{proof}

 Next result is a consequence of the fact that the biorthogonal functionals to the basis $ \mathcal{B}_n$ are elements that belong to the unit sphere of the dual of $\ell_\infty ^n$.

\begin{lemma}
\label{cer-co}
If $\varepsilon>0$, $x,y \in \co(\mathcal{B}_n)$ satisfies that 
$$x=\sum_{i=1}^{n} \alpha_iv_i^{n}, \sep y=\sum_{i=1}^{n} \beta_iv_i^{n} \sep \text{and} \sep \Vert x-y\Vert < \varepsilon,$$
then  $$\max \{|\alpha_i - \beta_i| : i\le n\}<\varepsilon.$$
\end{lemma}

\begin{proof}
By Lemma \ref{B-base} we have that  
$$
\alpha_1=\frac{1+x(2)}{2}, \sep \beta_1=\frac{1+y(2)}{2}, \sep  \alpha_n=\frac{1-x(n)}{2} ,  \sep \beta_n=\frac{1-y(n)}{2}
$$
and
$$
\alpha_i=\frac{x(i+1)-x(i)}{2}, \sep   \beta_i=\frac{y(i+1)-y(i)}{2} \sem \text{for each } \sep 1 <i<n. 
$$
As a consequence we obtain that 
\begin{eqnarray*}
|\alpha_1-\beta_1|&=& \bigg | \frac{x(2)-y(2)}{2}\bigg | \le  \frac{  \Vert x-y \Vert }{2}  < \frac{\varepsilon}{2},  \\
|\alpha_n-\beta_n|& <& \bigg | \frac{y(n)-x(n)}{2}\bigg | < \frac{\varepsilon}{2} \sem \text{and} \\
|\alpha_i-\beta_i|&=& \bigg | \frac{x(i+1)-y(i+1)}{2} + \frac{y(i)-x(i)}{2} \bigg | < \frac{\varepsilon}{2} + \frac{\varepsilon}{2} = \varepsilon \sep \text{if} \sep 1 < i < n            .
\end{eqnarray*}
Hence the proof is finished.
\end{proof}

In order to prove 
	the main result
 condition 4) in Proposition \ref{pro-char} allows to show  easily that a space $Y$ with the AHSp-$\ell_\infty ^n$ also satisfies that the pair $(\ell_\infty ^n,Y)$ has the BPBp. Proving the   converse
 is more delicate.  To this purpose  we  use again the same  reformulation of the  AHSp-$\ell_\infty ^n$, but 
in this case    Lemma \ref{le-close-face} also plays  an essential role.

\begin{theorem}
\label{teo-char}
    Let $Y$ be a Banach space.  The pair $(\ell_\infty^n,Y)$ has the  Bishop-Phelps-Bollob{\'a}s property for operators if and only if $Y$ has the approximate  hyperplane  sum property for $\ell_\infty^n$.
    \newline
    Moreover, if $(\ell_\infty ^n, Y) $ satisfies Definition \ref{def-BPBp} with the function $\eta_n $, then $Y$ has the AHSp-$\ell_\infty ^n $ with $\gamma_n (\varepsilon)= \eta_n \bigl( \frac{\varepsilon}{n(n+1)}\bigr) $. In case that  $Y$ has the AHSp-$\ell_\infty ^n $ for the function $\gamma_n $ (see Definition  \ref{def-AHSPLn}), the pair
    	    	 $(\ell_\infty ^n, Y) $ satisfies BPBp with the function $\eta_n (\varepsilon)= \gamma_n^2 \bigl( \frac{\varepsilon }{n+1} \bigr) $. 
  \end{theorem}
\begin{proof}
Assume that the pair $(\ell_\infty ^n, Y)$ has the  BPBp and let us fix $0<\varepsilon <1$. Now we define $\nu_n(\varepsilon)= \eta_n\big(\frac{\varepsilon}{n+1}\big)$,
where $\eta_n\big(\frac{\varepsilon}{n+1}\big)<\frac{\varepsilon}{n+1}$ is  the positive real number satisfying the BPBp for  $\frac{\varepsilon}{n+1}$.

Let  $(y_i )_{i \le n} \in M_Y^n$  and assume that
$\sum_{i=1}^{n} \alpha_i y_i $ is convex combination such that

$$
\Big\Vert  \sum_{i=1}^{n} \alpha_iy_i \Big\Vert  >  1 - \nu_n(\varepsilon).
$$
In view of Proposition  \ref{prop-ident}, there is a unique operator $T$  in  $B_{L(\ell_\infty ^n,Y)}$  such that $y_i= T(v_{i}^n)$  for each $1 \le i \le n$.  The element $x_0= \sum_{i=1}^{n} \alpha_i v_i^{n}$ satisfies that  $x_0 \in
S_{\ell_\infty^n}$, and by assumption we know that
\begin{equation}
\label{Tx0}
\bigl\Vert T(x_0) \bigr\Vert = \biggl \Vert \sum_{i=1}^{n} \alpha_i y_i   \biggr
\Vert > 1 - \eta_n\Big(\frac{\varepsilon}{n+1}\Big)> 1 - \frac{\varepsilon}{n+1} >   0.
\end{equation}

Since the pair $(\ell_\infty ^n, Y)$ has the BPBp, there are $u_0\in S_{\ell_\infty^n}$  and  $S\in S_{L(\ell_\infty^n,Y)}$  satisfying the following conditions
\begin{equation}
\label{Su0}
\Vert S(u_0)\Vert =1,  \sem \Vert u_0 - x_0\Vert <\frac{\varepsilon}{n+1}  \sem
\text{and}  \sem \Big \Vert S-\frac{T}{\Vert T\Vert } \Big \Vert
<\frac{\varepsilon}{n+1}.
\end{equation}

By Lemma \ref{le-close-face}, it may be assumed that $u_0$ is in $\co(\mathcal{B}_n)$. As a consequence, there exists
\linebreak[4]
$\{\beta_i: 1\le i \le n\} \subset \R_0^{+}$ such that 
$$
u_0= \sum_{i=1}^{n}\beta_iv_i^{n} \sem \text{and}  \sem \sum_{i=1}^{n}\beta_i =1.
$$
Now we define the sets $A= \{ i\in \{1,\ldots, n\} : \beta_i \neq 0\}$ and $A'=\{1,\ldots, n\} \setminus A$.  By (\ref{Su0}) and in view of Lemma \ref{cer-co}, we obtain that
\begin{equation}
\label{sum-alphai-A}
\sum_{i\in A}\alpha_i= 1- \sum_{i\in A'}\alpha_i = 1- \sum_{i\in A'} |\alpha_i-\beta_i| \ge 1
-\frac{\varepsilon}{n+1} \vert A' \vert  > 1 - \varepsilon.
\end{equation}

Finally we check that $(z_i)_{i \le n}=(S(v^n_i))_{i \le n}$ is the desired element in $ M_Y^n$. Clearly  $(z_i)_{i \le n} \in M_Y^n$  since $S\in S_{ L(\ell_\infty^n, Y)}$ (see Proposition \ref{prop-ident}).  Also, for each $1\le i\le n$, we have that

\begin{eqnarray}
\label{zi,yi}
\Vert  z_i - y_i\Vert  & = &\Vert  S(v_i^n)- T(v_i^n) \Vert   \\
\nonumber
 & \leq& \Big\Vert  S(v_i^n) - \frac{T}{\Vert T\Vert}(v_i^n)  \Big\Vert  + \Big\Vert
\frac{T}{\Vert T\Vert }(v_i^n) - T(v_i^n)  \Big\Vert        \\
\nonumber
 & \leq &\Big\Vert  S -  \frac{T}{\Vert T\Vert }  \Big\Vert  + 1 - \Vert T\Vert
  \\
\nonumber
& < &\frac{\varepsilon}{n+1} +  \frac{\varepsilon}{n+1} \le \varepsilon  \sem \text{(by \eqref{Su0} and  \eqref{Tx0})}.
\end{eqnarray}

By using Hahn-Banach Theorem and  \eqref{Su0},  there is an element $y^* \in S_{Y^*}$ such that $y^*(S(u_0))=1$, that is, 
$$
y^*\bigg(\sum_{i=1}^{n}\beta_i z_i\bigg)=1.
$$
Then,  it is clear that $y^*(z_i)=1$  for each $i\in A$  and  we obtain that 
$$
 \vert A \vert =  y^*  \Bigl( \sum_{i\in A}   z_i  \Bigr) \le  \bigg\Vert  \sum_{i\in A}z_i   \bigg\Vert  \le  \sum_{i\in A}\Vert z_i \Vert \le  \vert A \vert.
$$
As a consequence  $\Vert  \sum_{i\in A}z_i \Vert=\vert A \vert$. In view of  \eqref{sum-alphai-A} and \eqref{zi,yi} we proved  that $Y$ satisfies condition 4) in 
Proposition \ref{pro-char}  for $\nu_n(\varepsilon)= \eta_n\big(\frac{\varepsilon}{n+1}\big)$. Hence $Y$ has the AHSp-$\ell_\infty ^n$ for the function  $\gamma _ n (\varepsilon) = \eta _n\big( \frac{ \varepsilon}{ n(n+1)} \big)$.

Assume that  $Y$ satisfies the  AHSp-$\ell_\infty ^n$. Let  $0<\varepsilon <1$ and
we  write $\eta_n(\varepsilon)= \nu_n(\frac{\varepsilon}{n+1})$, where
$\nu_n(\frac{\varepsilon}{n+1})$ is the positive real number satisfying condition 4)
in Proposition \ref{pro-char} for  $\frac{\varepsilon}{n+1}$.

Assume that $T\in S_{L(\ell_\infty^n,Y)}$ and  $x_0 \in
S_{\ell_\infty^n}$  are such that
$$
\|T(x_0)\| >1-\eta_n (\varepsilon).
$$
Up to a linear surjective isometry on $\ell_\infty ^n$ we can assume that $x_0 \in E_1^n$.                 
Let $\sigma: \{2,\ldots, n\} \rightarrow \{2,\ldots, n\}$ be a bijection such that 
$$x_0(\sigma(i))\le x_0(\sigma(i+1)) \sep \forall i \in \{2,\ldots, n-1\},$$
and $S_\sigma$  the linear isometry  on  $\ell_\infty^n$ given by
$$S_\sigma(x)(1)=x(1), \sep S_\sigma(x)(i)=x(\sigma(i)) \sep \forall i \in \{2,\ldots, n\}.$$
By using $T\circ S_\sigma$ instead of $T$, if needed, we can also  assume, without loss of generality, that $x_0\in O_n$. By Lemma \ref{coB=O} we know that $x_0\in\co (\mathcal{B}_n)$, then there exists $\{\alpha_i: 1\le i \le n\} \subset \R_0^{+}$  such that
$$
x_0=\sum_{i=1}^{n}\alpha_iv_i^n \sem \text{and}  \sem \sum_{i=1}^{n}\alpha_i =1.
$$

In view of Proposition \ref{prop-ident} the element $(y_i)_{i \le n}= (T(v_i ^n))_{i \le n} \in M_Y^n$. By assumption we know that $$
\biggr \Vert \sum_{i=1}^{n}\alpha_i y_i \biggl\Vert  = \Vert T (x_0) \Vert  >1-\nu_n\Big(\dfrac{\varepsilon}{n+1}\Big).
$$
By assumption  there is a (nonempty) set  $A\subset
\{1,\ldots, n\}$  and   $(z_i)_{i \le n} \in M_Y^n$  such that
\begin{equation}
\label{sum-A-alphai}
\sum_{i\in A}\alpha_i>1-\frac{\varepsilon}{n+1}>0, \seg
\|z_i-y_i\|<\frac{\varepsilon}{n+1},  \sep \forall i \le n
\end{equation}
and ´ it is also satisfied
\begin{equation}
\label{A-zi-sum}
\Bigl\Vert \sum _{ i \in A} z_i \Bigr \Vert = \vert A \vert .
\end{equation}

Let $S$ be the unique  element in  $L(\ell_{\infty}^n,Y)$  satisfying that 
$S(v_i^n)=z_i$ for each $1\leq i \leq n$.   In view of \eqref{sum-A-alphai} we can use  Proposition  \ref{prop-ident}  to obtain  that $S \in B_{L(\ell_\infty^n,Y)}$  and  $\| S- T \| < \varepsilon$. The element  $u_0$ given by $u_0=\sum_{i\in A}\frac{\alpha_i}{\sum_{i\in A}\alpha_i}v_i^n$ belongs to $S_{\ell_{\infty}^n}$.  By  \eqref{A-zi-sum} the operator $S$
attains its norm at $u_0$ since
$$
1  =  \frac{ \Vert \sum _{i \in A} \alpha _i z_i \Vert}{  \sum _{i \in A} \alpha _i
}   = \Vert S (u_0) \Vert \le \Vert S \Vert \le 1.
$$
As a consequence  $S \in S_{L(\ell_\infty^n,Y)}$.  If we write $A'=\{1,\ldots, n\} \setminus A$,  we obtain that
\begin{align*}
 \Vert  u_0 - x_0\Vert  & =  \biggl\Vert \sum_{i\in
A}\dfrac{\alpha_i}{\sum_{i\in A}\alpha_i}v_i^n-   \sum_{i=1}^{n}\alpha_i v_{i}^n 
 \bigg\Vert \\
& =  \bigg \| \Big(1-\frac{1}{\sum_{i\in A}\alpha_i}\Big) \sum_{i\in A}\alpha_iv_i^n + \sum_{ i \in A'  }\alpha_i v_{i}^n\bigg \| \\
& \leq \bigg(\frac{1}{\sum_{i\in A}\alpha_i}-1\bigg) \sum_{i\in A}\alpha_i + \sum_{i \in A'} \alpha_i \\
& =2\sum_{i \in A'} \alpha_i \\
&< 2\frac{\varepsilon}{n+1} \le \varepsilon \sem \text{(by \eqref{sum-A-alphai})}.
\end{align*}
We proved that the pair $(\ell_\infty^n,Y)$ has the BPBp with $\eta_n(\varepsilon)= \nu_n(\frac{\varepsilon}{n+1})$.

Lastly, in case that $Y$ satisfies  AHSp-$\ell_\infty ^n$ with the function  $\gamma _n$ we know that  $Y$ satisfies condition 4) in  Proposition \ref{pro-char} for the function $\nu_n = \gamma _n ^2$.  As a consequence of the  above  proof, we deduce that  the pair $(\ell_\infty^n,Y)$ has the BPBp with $\eta_n(\varepsilon)= \gamma _n ^2 \big( \frac{ \varepsilon}{n+1} \big)$. 
\end{proof}

\section{Examples of spaces with the approximate hyperplane sum property for $\ell_\infty^n$}

As a consequence of the characterization stated in Theorem \ref{teo-char} and  known  results  related to the Bishop-Phelps-Bollob\'{a}s property for operators, 
the following classes of Banach spaces have the  AHSp-$\ell_\infty^n$, for any natural number $n$.

\begin{itemize}
	\item  Finite-dimensional spaces (see \cite[Proposition 2.4]{AAGM1}).
	\item  Uniformly convex spaces in view of \cite[Theorem 5.2]{AAGM1}.
	\item Spaces with the property $\beta $ of Lindenstrauss (see \cite[Definition 2.1 and Theorem 2.2]{AAGM1}).
	\item  Uniform algebras  (\cite[Theorem 3.6]{CGK}).
	\item  $C_0(L,Y)$, for any locally compact  Hausdorff space $L$, whenever $Y$ is a space with the AHSp-$\ell_\infty ^n$  (see the argument of \cite[Proposition 2.4]{ABGKM2}).
\end{itemize}

Moreover there is also a nontrivial class of Banach spaces containing  uniformly convex spaces and spaces with the property $\beta $ (of Lindenstrauss) satisfying also the previous property 
(see \cite[Theorem 2.4]{AGKM}).

The  main result of this section states that $L_1 (\mu) $ also shares  the  AHSp-$\ell_\infty^n$, for any positive  measure $\mu$ and any natural number $n$ (see Corollary \ref{cor-L1-AHSP-lin}).
In order to  obtain such result the key idea is to prove this  statement  for  $\ell_1$.

Throughout this section we denote by   $u^* $
 the functional  on $\ell_{1} $ given by
$$
u^* (x)=\sum _{k=1}^\infty x(k) \seg (x \in \ell_1).
$$
On the space $\ell_1$ we consider the usual   order as a sequence space.

\begin{lemma}
	\label{le-auxiliar}	
	Let   $r,s \in \R^+$, $y \in \ell_1$,   $ m \in \N$ and  $\{ x_i : 1 \le i  \le m \}  \subset  \ell_{1}$. Assume that 
	$$
	1-r \le u^*(x_i+y)  \sem \text{and}\sem \Vert x_i+y \Vert \le 1+s  \sep \text{for all} \sep i\le m.
	$$
	Then there exists $w\in \ell_{1}$ such that
	$$
	w\ge y,  \seg
	\Vert w-y\Vert \le m(r+s)  \sem \text{and}   \sem  x_i+w \ge 0 \sep \text{for all}\sep i\le m.
	$$
\end{lemma}
\begin{proof} 
	The statement for $m=1$ is a consequence of \cite[Lemma 4.5]{ADS}.   By using this fact  we prove the general case. 
	Let $m \in \N$,  $y \in \ell_1$, $\{ x_i: i \le m\} \subset \ell_1$ and assume that  
	$$
	1-r \le u^*(x_i+y)  \sem \text{and}\sem \Vert x_i+y \Vert \le 1+s  \sep \text{for all} \sep i\le m.
	$$
	By the result for $m=1$, there is a subset $\{ w_i: i \le m \} \subset \ell_1$ such that 
	\begin{equation}
	\label{le-case-1}
	w_i\ge y,  \seg
	\Vert w_i-y\Vert \le r+s  \sem \text{and}   \sem  x_i+w_i \ge 0 \sep \text{for all}\sep i\le m.
	\end{equation}
	Take $w= \max \{ w_i: 1 \le i \le m\}$.  There  is a family of pairwise disjoint sets $\{A_i: i \le m\}   \subset \N$  such that $\N = \cup _{i \le m } A_i$
	and  satisfying also that
	\begin{equation}
	\label{w-wi-Ai}
	w \chi _{A_i} = w_i  \chi _{A_i} \sem  \text{for all} \sep i\le m.
	\end{equation}
	As a consequence $w= \sum _{ i=1}^m w_i \chi _{ A_i} \in \ell_1$ and in view of \eqref{le-case-1} it is  satisfied that
	\begin{equation}
\label{w-big}
w\ge w_i\ge  y \sem \text{and}     \sem  x_i+w \ge x_i + w_i  \ge 0 \sep \text{for all }\sep i\le m.
\end{equation}
Since $\{ A_i: i \le m\}$ is a partition of $\N$  we also have that
	\begin{eqnarray}
\label{w-close-y}
\Vert w-y \Vert  &=&   \sum _{i=1}^m  \Vert (w-y) \chi _{ A_i} \Vert  \notag\\
&=&  \sum _{i=1}^m  \Vert (w_i-y) \chi _{ A_i} \Vert  \sep \text{(by \eqref{w-wi-Ai})}  \notag \\
&\le &   \sum _{i=1}^m  \Vert (w_i-y)  \Vert  \\
& \le &  m(r+s)   \sep \text{(by \eqref{le-case-1})}.  \notag 
\end{eqnarray}
In view  of \eqref{w-big} and \eqref{w-close-y} the proof is finished.
\end{proof}

\begin{theorem}
\label{l1-AHSPLn}
For  each
 $n\in \N$ there exists a function $\rho_n: ]0,1[ \llll  ]0,1[$
 with the following properties 
\begin{enumerate}
\item[1)]  $\rho_n(\varepsilon) < \varepsilon$ for all  $\varepsilon \in ]0,1[$,
\item[2)]  For each  $0 < \varepsilon < 1$,  $ (y_i)_{i \le n}  \in M_{\ell_{1}}^n$ and $n_0\le n$  
 such that
 $$
 u^* (y_i) > 1 - \rho_n (\varepsilon) \sep  \forall i \le n_0,
 $$ 
 there exists $ (z_i)_{i \le n}  \in M_{\ell_{1}}^n$   
satisfying the following two conditions
\begin{itemize}
\item[a)] $ \Vert z_i - y_i \Vert < \varepsilon$, for all $i \le n$  \  and
\item[b)] $z_i\ge 0$ and $u^*(z_i)=1$,  for all  $i \le n_0.$
\end{itemize}
\end{enumerate}
\end{theorem}
\begin{proof}
We define inductively the function $\rho_n$ as follows.
 %
 For $n=1$ we check that the function $\rho_1: ]0,1[ \llll \R^+$ defined by
 $\rho_1(\varepsilon)=\frac{\varepsilon}{2}$ satisfies condition 2).\\
Assume that $0<\varepsilon<1$ and $y_1 \in B_{\ell_{1}}$  satisfy that $u^*(y_1)>1 - \rho_1(\varepsilon)$. If we define  $P_1$ and $a_1$ by 
$$
P_1=\{ k \in \N : y_1(k) \ge 0\}  \seg \text{and} \seg a_1 = y_1 \chi_{P_1},
$$
then  we  clearly have that
$$
1-\rho_1(\varepsilon) < u^*(y_1) \le  u^*(a_1) =\Vert a_1 \Vert \le \Vert y_1\Vert \le 1 \sep \text{and} \sep a_1 \ge 0.
$$
As a consequence  
$$
\Vert a_1-y_1 \Vert = \Vert y_1\chi_{\N \setminus P_1}\Vert=\Vert y_1\Vert - \Vert a_1\Vert  \le 1 - \Vert a_1 \Vert < \rho_1(\varepsilon).
$$
Now we define $z_1= a_1 + (1-\|a_1\|)e_1$.  So  it is satified  $z_1\ge 0$, $\|z_1\|=1$ and
$$
\|z_1-y_1\|\le \|a_1 - y_1\| + 1-\|a_1\| < \rho_1(\varepsilon) + \rho_1(\varepsilon) = \varepsilon.
$$
 Since $z_1 \ge 0$ and $\Vert z_1 \Vert =1$ we also have that  $u^*(z_1)=1$, so we proved the assertion stated for $n=1$.

 Assume that  $n$ is a natural number  and there is a function $\rho _n$ for which the statement holds true. We define  $\rho_{n+1} $  by
$$
\rho_{n+1}(\varepsilon)=\rho_{n}\bigg( \frac{\varepsilon}{8(n+2)|\mathcal{P}_{n}|}\bigg) \sem (\varepsilon \in ]0,1[).
$$
By assumption the function  $\rho_{n+1}: ]0,1[ \llll ]0,1[$ satisfies (1).

Now we prove that  condition 2) is also  satisfied. Assume that $0<\varepsilon<1$, $ 1 \le n_0 \le n+1$ and the element  $(y_i)_{i \le n+1}  \in M_{\ell_{1}}^{n+1}$ satisfy that
$$ 
u^* (y_i) > 1 - \rho_{n+1} (\varepsilon), \sem  \forall i \le n_0.
$$ 
Since  $(y_i)_{i \le n+1} \in M_{\ell_{1}}^{n+1}$  we know that $(y_i)_{i \le n}$ is an element in  $M_{\ell_{1}}^{n}$. By assumption there  is an element $(b_i)_{i \le n} \in M_{\ell_{1}}^{n}$  satisfying   the following two properties 
\begin{eqnarray}
\label{bi-yi}
\|b_i-y_i\| < \frac{\varepsilon}{8(n+2)|\mathcal{P}_{n}|},\sep \sep \forall i\le n 
\end{eqnarray}
and
\begin{eqnarray}
\label{bi-u*}
b_i\ge 0\sep \text{and} \sep u^*(b_i)=1 \sep \sep  \forall i \le \min \{n_0, n\}.
\end{eqnarray}
For each $1 \le i \le n $ we define 
$$
z_i= \frac{b_i}{1 +\frac{\varepsilon}{8}} + \bigg(1 - \frac{1}{1 +\frac{\varepsilon}{8}}\bigg)e_1.
$$
Now we  check that $(z_i)_{i \le n} \in M_{\ell_{1}}^{n}$. If $(i_1, \ldots , i_k) \in \mathcal{I}_n$, since  $(b_i) _{ i \le n} \in  M_{\ell_{1}}^{n}$,   then we  obtain that  
\begin{eqnarray}
\label{suma-n-z}
 \biggl\Vert  \sum_{j=1}^{k}(-1)^{j+1}z_{i_j}\biggl\Vert  
 & = & \biggl\Vert \frac{1}{1 +\frac{\varepsilon}{8}} \sum_{j=1}^{k}(-1)^{j+1}b_{i_j} + \bigg(1 - \frac{1}{1 +\frac{\varepsilon}{8}}\bigg)e_1 \bigg\Vert \\
 \nonumber
& \le&  \frac{1}{1 +\frac{\varepsilon}{8}} + 1 - \frac{1}{1 +\frac{\varepsilon}{8}} =1.
\end{eqnarray}
From \eqref{bi-u*} we also  have that  
$$
z_i\ge 0 \sep \text{and} \sep u^*(z_i)=1, \sep \sep \forall i \le \min \{n_0, n\}.
$$
On the other hand, if $i\le n$ then
\begin{eqnarray}
\label{zi-yi}
\|z_i- y_i\|
& \le& \bigg \|z_i - \frac{b_i}{1+ \frac{\varepsilon}{8}} \bigg \| + \bigg \| \frac{b_i}{1+ \frac{\varepsilon}{8}} - b_i \bigg \| + \|b_i-y_i\| \\
\nonumber
& <&  2\bigg(1 - \frac{1}{1 +\frac{\varepsilon}{8}}\bigg) +  \frac{\varepsilon}{8(n+2)|\mathcal{P}_{n}|} \sem \text{(by \eqref{bi-yi})}\\
\nonumber
& <& \frac{\varepsilon}{4}+ \frac{\varepsilon}{8} < \varepsilon.
\end{eqnarray}

It only remains to define a vector $z_{n+1}\in B_{\ell_{1}}$ such that  $ (z_i)_{i \le n+1}  \in M_{\ell_{1}}^{n+1}$, $\|z_{n+1}- y_{n+1}\| < \varepsilon$ and  in the case  that $n_0=n+1$ we also need that  conditions  $z_{n+1} \ge 0$ and  $u^*(z_{n+1})=1$ are  satisfied. For this  last step we consider the following  two cases.

\noindent 
 \underline{Case 1}: Assume that $n_0 <n+1$.\\
In such case  we define $z_{n+1}=\frac{y_{n+1}}{1 +\frac{\varepsilon}{8}}$. Since $y_{n+1} \in B_{\ell_1}$ it is clear that $\|z_{n+1}\|\le 1$ and
$$
\|z_{n+1}-y_{n+1}\|\le 1 - \frac{1}{1+\frac{\varepsilon}{8}}< \varepsilon.
$$
In view of  \eqref{suma-n-z}, having in mind that $z_{n+1}\in B_{\ell_1}$, in order to prove  that $(z_i)_{i \le n+1} \in M_{\ell_1}^{n+1}  $,  
	 it suffices to show that the condition defining $M_{\ell_1}^{n+1} $  is satisfied for linear combinations  containing at least three elements including  $z_{n+1}$. So  let us fix  $(i_1, \ldots , i_k) \in \mathcal{P}_{n}$.  By using that $(y_i)_{i \le n+1}\in M_{ \ell_1}^{n+1} $ we obtain that  
\begin{align*}
 \biggl\Vert  \sum_{j=1}^{k}(-1)^{j+1}z_{i_j} + z_{n+1}\biggl\Vert  
 & =  \biggl\Vert \frac{1}{1 +\frac{\varepsilon}{8}} \sum_{j=1}^{k}(-1)^{j+1}b_{i_j} +  \frac{y_{n+1}}{1 + \frac{\varepsilon}{8}} \bigg\Vert \\
& \le \frac{\|  \sum_{j=1}^{k}(-1)^{j+1}(b_{i_j}-y_{i_j})\| + \Vert  \sum_{j=1}^{k}(-1)^{j+1}y_{i_j}+ y_{n+1}\|}{1+\frac{\varepsilon}{8}} \\
& < \frac{\frac{\varepsilon k}{8(n+2)|\mathcal{P}_{n}| } +1 }{1+\frac{\varepsilon}{8}} < 1\sem \text{(by \eqref{bi-yi})}.
\end{align*}

So  the proof is finished in case 1.

\noindent 
\underline{Case 2}: Assume that  $n_0 =n+1$.
Let define   $P$ and $a$ by 
$$
P=\{ k \in \N : y_{n+1}(k) \ge 0\}  \seg \text{and} \seg a = y_{n+1} \chi_{P}.
$$
By assumption $u^*(y_{n+1} ) >  1-\rho_{n+1}(\varepsilon) $, so  we have that
\begin{eqnarray}
\label{a-y_n+1}
\sem 
a \ge 0, \sep  1-\rho_{n+1}(\varepsilon)<u^*(a) = \|a\| \le 1 \sep \text{and} \sep \|a-y_{n+1}\|<\rho_{n+1}(\varepsilon).
\end{eqnarray}
Note that for each $(i_1, \ldots , i_k) \in \mathcal{P}_{n}$, by using \eqref{bi-u*}  and \eqref{a-y_n+1} we obtain that  
\begin{eqnarray}
\label{u*a}
1- \frac{\varepsilon}{8(n+2)|\mathcal{P}_{n}|}  <  u^*(a)= u^* \bigg( \sum_{j=1}^{k}(-1)^{j+1}b_{i_j} + a \bigg)
\end{eqnarray}
and
\begin{align*}
\bigg \| \sum_{j=1}^{k}(-1)^{j+1}b_{i_j} +a \bigg \| 
&\le \bigg \| \sum_{j=1}^{k}(-1)^{j+1}(b_{i_j} - y_{i_j})\bigg \|  +\bigg \| \sum_{j=1}^{k}(-1)^{j+1}y_{i_j} +y_{n+1}\bigg \|  \\
& \sep  \sep + \|a- y_{n+1}\| \\
 &<  1 +  (n+1)\frac{\varepsilon}{8(n+2)|\mathcal{P}_{n}|}\sem \text{(by \eqref{bi-yi} and \eqref{a-y_n+1})}.
\end{align*}
In view of \eqref{u*a} and the previous inequalities we can apply Lemma \ref{le-auxiliar}. Hence  there is $w\in \ell_1$ satisfying the following conditions 
\begin{eqnarray}
\label{w-a}
w \ge a, \sep \sep \|w-a\| \le |\mathcal{P}_{n}|(n+2)\frac{\varepsilon}{8(n+2)|\mathcal{P}_{n}|}= \frac{\varepsilon}{8}\\
\label{b+w-m-0}
\sum_{j=1}^{k}(-1)^{j+1}b_{i_j}+w\ge 0 \sep \sep \text{for each} \sep \sep (i_1, \ldots , i_k) \in \mathcal{P}_{n}.
\end{eqnarray}

 By \eqref{a-y_n+1} $a\ge0$, so  in view of \eqref{a-y_n+1}  and \eqref{w-a} we have that 
\begin{equation}
\label{sum-bij-w-pos}
1- \frac{\varepsilon}{8(n+2)|\mathcal{P}_{n}|}< \|a\|\le \|w\| \le \|w-a\| + \|a\| \le \frac{\varepsilon}{8} + 1.
\end{equation}
 Hence
$$
\frac{1 - \frac{\varepsilon}{8(n+2)|\mathcal{P}_{n}|}}{1 + \frac{\varepsilon}{8}}< \frac{\|w\|}{1+\frac{\varepsilon}{8}}\le1,
$$
so
\begin{equation}
\label{1-norma-w}
0\le1-\frac{\|w\|}{1+\frac{\varepsilon}{8}}< 1 - \frac{1 - \frac{\varepsilon}{8(n+2)|\mathcal{P}_{n}|}}{1 + \frac{\varepsilon}{8}}< \frac{\varepsilon}{8} + \frac{\varepsilon}{8(n+2)|\mathcal{P}_{n}|}.
\end{equation}
Finally we define  
$$
z_{n+1}=\frac{w}{1+\frac{\varepsilon}{8}} + \bigg( 1-\frac{\|w\|}{1+\frac{\varepsilon}{8}} \bigg)e_1.
$$
 Since $w \ge 0$, in view of \eqref{1-norma-w}, it  is clear that $z_{n+1}\ge 0$ and $u^*(z_{n+1})=1$.  Now we check that $z_{n+1}$  is close to $y_{n+1}$ as follows 
\begin{align}
\label{zn+1-yn+1}
\nonumber
\|z_{n+1}-y_{n+1}\| 
&\le \bigg\|z_{n+1}- \frac{w}{1+\frac{\varepsilon}{8}} \bigg\|  +   \bigg\| \frac{w}{1+\frac{\varepsilon}{8}} - w\bigg\| + \|w-a\| + \|a-y_{n+1}\| \\
& <   \frac{\varepsilon}{4} + \frac{\varepsilon}{4(n+2)|\mathcal{P}_{n}|} +   \bigg( 1-\frac{1}{1+\frac{\varepsilon}{8}}\bigg) \|w\|   \sep \text{(by  \eqref{a-y_n+1}, \eqref{w-a} and \eqref{1-norma-w})} \\
\nonumber
&<  \frac{\varepsilon}{2}+ \frac{\varepsilon}{8} < \varepsilon \sem \text{(by \eqref{sum-bij-w-pos})}.
\end{align}

Lastly, if $(i_1, \ldots , i_k) \in \mathcal{P}_{n}$,
from \eqref{b+w-m-0} and  \eqref{1-norma-w} and by using also \eqref{bi-u*}  and  $w \ge 0$ 
 we  obtain that
\begin{align}
\label{sum-z-zn+1}
\nonumber
\bigg\| \sum_{j=1}^{k}(-1)^{j+1}z_{i_j} + z_{n+1}\bigg \| 
&= \bigg\| \frac{1}{1+\frac{\varepsilon}{8}}\bigg( \sum_{j=1}^{k}(-1)^{j+1}b_{i_j} + w\bigg) +   \bigg ( 1- \frac{\|w\|}{1 +\frac{\varepsilon}{8}} \bigg)e_1 \bigg\|\\
& = \frac{1}{1+\frac{\varepsilon}{8}}u^* \bigg(   \sum_{j=1}^{k}(-1)^{j+1}b_{i_j} + w   \bigg) +   1- \frac{\|w\|}{1 +\frac{\varepsilon}{8}}
 \\
\nonumber
 &= \frac{1}{1+\frac{\varepsilon}{8}}u^*(w) + 1 - \frac{u^*(w)}{1+\frac{\varepsilon}{8}}=1.
\end{align}
Since $(z_i)_{i \le n} \in M_{\ell_{1}}^{n}$ and by  \eqref{suma-n-z} and \eqref{sum-z-zn+1}, we conclude that $(z_i)_{i \le n+1} \in M_{\ell_{1}}^{n+1}$.
 In view of \eqref{zi-yi} and \eqref{zn+1-yn+1}, the proof is also finished  in Case 2.

As a consequence, we showed  that  $\ell_{1}$ verifies the statement for $n+1$ with $\rho_{n+1}(\varepsilon)=\rho_{n}\big( \frac{\varepsilon}{8(n+2)|\mathcal{P}_{n}|}\big)$.
\end{proof}

\begin{corollary}
	\label{co-ell1-AHSP-n}
The space $\ell_1$ has the      approximate  hyperplane  sum property for $\ell_\infty^n$, for any positive integer $n$.

Indeed   for each $n \in \N$ there is a function $\gamma_n$ such that $\ell_1$ and $\ell_1 ^m$ has the   approximate  hyperplane  sum property for $\ell_\infty^n$ with the function $\gamma _n$ for any natural number $m$.  We denoted by $\ell_1 ^m$  the linear space $\R^m$ endowed with the $\ell_1$-norm.
\end{corollary}

\begin{proof}
It suffices to show that $\ell_1$ satisfies condition 3) of Proposition \ref{pro-char}. The same argument can be applied  to $\ell_1 ^m$ since  the proof of
Theorem \ref{l1-AHSPLn} is also valid for this space, for any natural number $m$. In the last case we obtain  additionally that there is  a function  $\rho _n$ satifying  Definition \ref{def-AHSPLn}  for the space
$\ell_1 ^m$  that does not depend on $m$.

Assume that the function $\rho_n$ satisfies Theorem \ref{l1-AHSPLn}. We show that the same function also satisfies condition 3) in Proposition \ref{pro-char} for the $1$-norming set $B=\ext(B_{\ell^*_1})$.\\
Assume that $(y_i)_{i \le n} \in M_{\ell_{1}}^{n}$, $n_0\le n$ and $y^*\in B$ satisfy $y^* (y_i) > 1 - \rho_n (\varepsilon) $ for each $i \le n_0$.  It is well known and inmediate that $\ext(B_{\ell^*_1})=\{v^* \in \ell^*_1 :  \vert v^*(e_n) \vert=1 ,  \sep \forall   n\in \N  \}$. So if we define   $\varepsilon_n=\text{sign}(y^*(e_n))$ then $ \varepsilon  \in \{1,-1\} $ for each $n\in \N$.\\ As a consequence, the mapping $T: \ell_{1} \llll \ell_{1}$ given by
$$
T((x_n))=(\varepsilon_nx_n) \sep  ((x_n)\in \ell_{1})
$$
is a linear surjective isometry on $\ell_{1}$ satisfying $T=T^{-1}$,  so $T^t=(T^t)^{-1}$.\\
Hence $T^t(y^*)(e_n)=y^*(T(e_n))=y^*(\varepsilon_ne_n)=1$, for each $n\in \N$, so $T^t(y^*)=u^*$; that is, $T^t(u^*)=y^*$.
So  $u^*(T(y_i))=T^tu^*(y_i)=y^*(y_i)>1- \rho_n(\varepsilon)$, for each $i\le n_0$.
Since $(y_i)_{i \le n} \in M_{\ell_{1}}^{n}$ and $T$ is a linear isometry  we have that  $(T(y_i))_{i \le n} \in M_{\ell_{1}}^{n}$. By using Theorem \ref{l1-AHSPLn} 
 there is $(x_i)_{i \le n} \in M_{\ell_{1}}^{n}$ such that 
$$
\|x_i-T(y_i)\| < \varepsilon, \sep  \forall i \le n \sem \text{and} \sem u^*(x_i)=1, \sep  \forall i \le n_0.
$$ 
As a consequence 
$$
\|T(x_i)-y_i\| < \varepsilon, \sep  \forall i \le n \sem \text{and} \sem y^*(T(x_i))=1, \sep  \forall i \le n_0.
$$
Since $(x_i)_{i \le n} \in M_{\ell_{1}}^{n}$, the element $(T(x_i))_{i \le n} $ also belongs to $M_{\ell_{1}}^{n}$. We showed that condition 3) in Proposition  \ref{pro-char} holds for $\ell_1$.
\end{proof}

\vskip5mm

Next result generalizes \cite[Theorem 4.7]{ADS}. We include  a proof for the sake of completeness.

\begin{proposition}
\label{Y-ani}
Let $n\in\N$ and    $\Gamma _n : ]0,1[ \llll ]0,1[$ be a function. Assume that            $Y$ is  a Banach space such that $Y=\overline{\cup \{Y_\alpha : \alpha \in \Lambda  \}}$, where $\{Y_\alpha : \alpha \in \Lambda  \}$  is a nested family of subspaces of $Y$ satisfying uniformly the  AHSp-$\ell_\infty^n$ with the function $\Gamma_n$. Then $Y$ has the AHSp-$\ell_\infty^n$ with the function 
\linebreak[4]
$\gamma_n (\varepsilon) =  \Gamma_n \bigl( \frac{\varepsilon}{2} \bigr)$.
\end{proposition}

\begin{proof} 
Given $0<\varepsilon<1$, we define $\gamma _n ( \varepsilon )= \Gamma_n \bigl( \frac{\varepsilon}{2} \bigr).$ 
Assume that $(a_i)_{ i \le n} \in M_{Y} ^n $   and that for some nonempty   set   $A \subset \{1,\ldots, n\}$ and $y^{*} \in S_{Y^{*}}$, it is satisfied that
$$
y^{*}(a_i)>1-\gamma_n (\varepsilon) , \sep  \forall  i\in A.
$$
Let us choose a real number $t$ such that
$$
0<t< \frac{1}{n+1}\min  \Bigl \{ \frac{ \varepsilon}{2}, \min  \bigl\{ y^*(a_i)-1+ \gamma_n (\varepsilon): i\in A \bigr\} \Bigr\}.
$$ 
By assumption there exist $\alpha_0\in \Lambda$  and $\{b_i: i \leq n\} \subset
B_{Y} \cap Y_{\alpha_0}$ satisfying
$$
\Vert  b_i - a_i \Vert  < t, \sep\ \ \forall i\le n.
$$
Now we define $y_i=\frac{b_i}{1+ nt }$ for each $i\leq n$. By using that $(a_i) _{ i \le n} \in M_Y ^n $ it is immediate to check that $(y_i)_{i \le n} \in M_{Y}^n \cap Y_{\alpha _0}$.   
It is clear that 
\begin{equation}
\label{y*yi}
\Vert  y_i-a_i  \Vert  \le \Bigl \Vert  \frac{b_i}{1+nt} - b_i \Bigr \Vert +  \Vert b_i-a_i \Vert 
< (n+1) t <\frac{\varepsilon}{2}, \sep  \text{for all} \sep i \le n.
\end{equation}
For each $i \in A$ we obtain that
\begin{equation}
\label{Ny}
y^*(y_i) >y^*(a_i)-(n+1) t>1-\gamma_n (\varepsilon)>0.
\end{equation}

Now we define the element $z^* \in Y_{\alpha_0}^{*}$  by
$$
z^*(z)=y^*(z) \seg  (z \in Y_{\alpha_0}),
$$
that satisfies   $z^* \in B_{ Y_{\alpha _0}^* }$.  From  \eqref{Ny} we know that  $z^* \ne 0$ and we also have that 
$$
\frac{z^*}{\|z^*\|}(y_i)=\frac{y^*}{\|z^*\|}(y_i)\geq y^*(y_i)>1- \Gamma_n \Bigl( \frac{\varepsilon}{2} \Bigr), \sep  \text{for all } i\in A.
$$
Since we assume that $Y_{\alpha _0} $ has the AHSp-$\ell_\infty ^n$,  there is $(z_i)_{i\le n} \in  M_{Y_{\alpha_0}}^n$ such that
$$
\|z_i-y_i\|<\frac{\varepsilon}{2} ,  \sep \text{for all }\sep  i \le n \sem \text{and}
        \sem  \bigl \Vert \sum _{ i \in A} z_i \bigr\Vert  = \vert A \vert .
        $$
By using also   \eqref{y*yi} we deduce that
$$
\Vert z_i - a_i \Vert \le   \Vert  z_i -  y_i \Vert +
\Vert  y_i - a_i \Vert < \frac{\varepsilon}{2} + \frac{\varepsilon}{2} =
\varepsilon, \seg \forall i \le n.
$$
This finishes the proof. 
\end{proof}

 As a consequence of Corollary \ref{co-ell1-AHSP-n}   and  Proposition \ref{Y-ani} we obtain that every space $L_1 (\mu)$ satisfies the  AHSp-$\ell_\infty^n$ for each natural number $n$, and the function $\rho_n$ that verifies the property for each natural number $n$ does not depend on $\mu$. In view of the characterization given in Theorem \ref{teo-char} we obtain the following result.

\begin{corollary}
\label{cor-L1-AHSP-lin}
For  each $n\in\N$ there is a function $\gamma_n: ]0,1[ \llll  ]0,1[$  such that  $L_1 (\mu)$ satisfies Definition \ref{def-AHSPLn} with such function, for any positive measure $\mu$. Hence, the pair $(\ell_\infty^n, L_1 (\mu) )$ has the BPBp for operators. Moreover, for each positive integer $n$ there is a function $\eta_n $ such that the pair $(\ell_\infty^n, L_1 (\mu) )$ satisfies  Definition \ref{def-BPBp} for such function,   for any positive measure $\mu$.
\end{corollary}

\end{document}